\documentclass[10pt]{amsart}

\usepackage{amsmath}
\usepackage{latexsym}
\usepackage{amsfonts}
\usepackage{amssymb}
\usepackage{amscd}
\usepackage{color}
\usepackage{graphicx}
\usepackage{mathrsfs}
\usepackage{amsthm}

\theoremstyle{plain}
\newtheorem{theorem}{Theorem}[section]
\newtheorem{lemma}[theorem]{Lemma}
\newtheorem{remark}[theorem]{Remark}
\newtheorem{proposition}[theorem]{Proposition}

\numberwithin{equation}{section}

\newcommand{\N}{\mathbb{N}} 
\newcommand{\R}{\mathbb{R}} 
\newcommand{\B}{\mathbb{B}}
\newcommand{\DD}{\mathscr{D}}

\def\d{\delta}		 	  \def\e{\varepsilon}

\def\L{{\mathcal L}}

\def\({\left(}       \def\){\right)}

\newcommand{\eps}{\varepsilon}

\newcommand{\pe}[2]{\left\langle #1,#2 \right\rangle}
\newcommand{\<}{\left\langle} \renewcommand{\>}{\right\rangle}

\newcommand{\dvert}{\boldsymbol{\vert}\!\!\boldsymbol{\vert}}
\newcommand{\vnorm}[1]{\dvert #1\dvert}

\newcommand{\pnorm}[1]{\left|#1\right|}
\newcommand{\rnnorm}[1]{\left\|#1\right\|_\infty}

\DeclareMathOperator*{\Div}{div}
\DeclareMathOperator*{\Var}{var}

\DeclareMathOperator*{\supp}{supp}

\newcommand\restr[2]{{
  \left.\kern-\nulldelimiterspace 
  #1 
  \right|_{#2} 
}}

\newcommand{\ignore}[1]{}

\begin{document}

\title{Smirnov-decompositions of vector fields}

\author{Alberto Rodr\'iguez-Arenas}
\address{Instituto Universitario de Matem\'atica Pura y Aplicada,
 Universitat Polit\`ecnica de Val\`encia,
E-46022 Valencia, Spain}
\email{alrodar3@posgrado.upv.es}

\author{Jochen Wengenroth}
\address{Fachbereich IV -- Mathematik, Universit\"{a}t Trier, D-54286 Trier, Germany}
\email{wengenroth@uni-trier.de}

 \begin{abstract}
We give streamlined proofs of theorems of S.\ Smirnov about the decomposition of vector fields of measures into curves.
\end{abstract}

\maketitle




\section{Introduction}

An $n$-dimensional charge is a $\sigma$-additive function $\mu=(\mu_1,\ldots,\mu_n)$ from the Borel $\sigma$-algebra $\B^n$ to $\R^n$, 
the components of such a vector measure are thus real-valued signed measures. Every charge $\mu$ acts on \emph{test fields}
$\varphi=(\varphi_1,\ldots,\varphi_n)$ of bounded measurable functions as
\[
\<\mu,\varphi\> = \sum_{k=1}^n \int_{\R^n} \varphi_k d\mu_k.
\]
The aim of this note is to prove theorems of Smirnov \cite{Smi} characterizing charges which decompose as
\[
\mu=\int_{K_\ell} [\gamma] d\nu(\gamma)
\]
for a positive finite measure $\nu$ on the set $K_\ell$ of $1$-Lipschitz curves $\gamma:[0,\ell]\to\R^n$ where $[\gamma]$ is
the charge acting on test fields as the work done by the field along $\gamma$, 
\[
\<[\gamma],\varphi\>= \int_{[0,\ell]} \<\varphi(\gamma(t)),\dot\gamma(t)\> d\lambda^1(t).
\]
Such decompositions are crucial in geometric measure theory (charges are the \emph{1-currents} of order $0$) 
and in the theory of optimal transport
(where $\nu$ is a \emph{transport plan})
but combined with the Hahn-Banach theorem they also have direct applications in approximation theory. As an illustration we give a very simple proof
of a result of Havin and Smirnov that for a compact set $M$ like the Koch curve in which every rectifiable curve is constant and all 
$f,f_1,\ldots,f_n\in C(M)$ 
one can approximate the function $f$ uniformly on $M$ by some $\psi\in\DD(\R^n)$ such that $\partial_k\psi$ approximate $f_k$.

\medskip
During the first 10 years since its publication Smirnov's theorems did not find much attention. This changed when \cite{Smi} was mentioned by
Brezis and Bourgain in \cite{BoBr} and later on when the relevance for the theory of optimal transport was realized, see, e.g.,  Santambrogio's book
\cite{San}
(Smirnov's scientific laurels like the fields medal probably also boosted the interest in his early work).

Meanwhile, there are variants of Smirnov's theorems \cite{PaSt1,PaSt2} in the context of metric currents in the sense of Ambrosio and Kirchheim \cite{AmKi}
but the classical approach of Smirnov is still interesting. However, at least the authors of the present note had a very hard time
to understand the details of \cite{Smi} and we do not know of a streamlined presentation. In view of the importance of the results for
several mathematical fields we will thus try to give detailed comprehensible proofs in Smirnov's spirit.

The main ingredients are a version of Liouville's theorem that flows of divergence free vector fields are volume preserving,  approximation and
smoothing by convolution together with the Arzel\'a-Ascoli theorem, and a dimension trick to reduce to the divergence free case.

\section{Charges and curves}
We recall basic facts about charges which can be found, e.g., in Schechter's \emph{Handbook of Analysis and its Foundations} \cite[§29]{Sch}
(although the term \emph{charge} is used in a different sense there)
or for the case $n=2$ in Rudin's classic \emph{Real and Complex Analysis} \cite[chapter 6]{Rud}.

The variation measure of an $n$-dimensional  charge $\mu$ is defined for Borel sets $E\subseteq\R^n$ as
\[
   \vnorm{\mu}(E)=\sup \left\{\sum_{k\in\N}|\mu (E_k)|: E_k\subseteq E \text{ are pairwise disjoint Borel sets}\right\},
\]
the similarity of the symbols $\vnorm{\cdot}$ and the euclidean norm $|\cdot|$ should not cause confusion, Schechter denotes the variation measure
as $/\mu/$ whereas Smirnov writes $\|\mu\|$.
This is indeed a finite positive measure and $\Var(\mu)=\vnorm\mu (\R^n)$ is called the \emph{total variation} of $\mu$.

The Radon-Nikodym theorem yields a measurable density $f:\R^n\to\R^n$ with $|f(x)|=1$ for all $x\in\R^n$ and $\mu=f\cdot \vnorm\mu$, i.e.,
\[
    \mu (E)=\int _E f  d\vnorm{\mu}
\]
for all $E\in\B^n$ where the integral is defined component-wise. This \emph{polar decomposition} of $\mu$ suggests to define $\mu$-integrals as
$\int gd\mu = \int gf d\vnorm\mu$ and 
\[
\<\mu,\varphi\>=\int \<f,\varphi\>d\vnorm\mu
\]
for test fields $\varphi=(\varphi_1,\ldots,\varphi_n)$ with bounded measurable components (again, we use the same symbol for the action of $\mu$ and for the euclidean scalar product).

The regularity of finite Borel measures on $\R^n$ yields that $\mu$ is already determined by its action on fields $\varphi\in\DD(\R^n)^n$ of (classical) test functions and that 
for open sets $E\subseteq\R^n$
\[
\vnorm\mu(E)=\sup\{\<\mu,\varphi\>: \varphi\in\DD(E)^n \text{ with } |\varphi|\le 1\}.
\]
The components $\mu_k$ of a charge are distributions (of order $0$) on $\R^n$ and thus have distributional partial derivatives $\partial_j\mu_k$ defined
by $\<\partial_j\mu_k,\psi\>=-\<\mu_k,\partial_j \psi\>$ for $\psi\in\DD(\R^n)$. 
The divergence of $\mu$ is $\Div(\mu)=\partial_1\mu_1+\cdots+\partial_n\mu_n$ 
so that for $\psi\in\DD(\R^n)$
\[
\<\Div(\mu),\psi\>=-\<\mu,\nabla\psi\> 
\]
for the gradient $\nabla\psi\in\DD(\R^n)^n$.

For the interpretation of charges as elements of the dual of $\DD(\R^n)^n$ it is enough that $\mu$ is a \emph{local charge}, i.e., $\mu$ is a $\sigma$-additive
$\R^n$-valued function on the ring of all bounded Borel subsets of $\R^n$. The same definition of $\vnorm\mu$ as for charges then yields
a locally finite variation measure. In particular, for every locally finite positive Borel measure $\rho$ on $\R^n$ and every field
$\phi=(\phi_1,\ldots,\phi_n)$ with locally $\rho$-integrable components $\phi\cdot\rho$ is a local charge. By the polar decomposition,
every local charge is of this form. 

\medskip
Every $1$-Lipschitz curve $\gamma:[a,b]\to\R^n$, i.e., $|\gamma(s)-\gamma(t)|\le |s-t|$ for all $s,t\in [a,b]$, yields a charge $[\gamma]$ which acts on
test fields as a curve integral
\[
\<[\gamma],\varphi\>=\int_{[a,b]}\<\varphi\circ\gamma,\dot\gamma\>d\lambda^1
\]
where $\lambda^1$ is the Lebesgue measure and  $\dot\gamma$ denotes the derivative of $\gamma$ which exists $\lambda^1$-almost everywhere
by Rademacher's theorem. The components of $[\gamma]$ considered as a vector measure are the image measures or push-forwards of
$\dot\gamma_k \cdot \lambda^1$ under $\gamma$, i.e.,
\[
\gamma_\#(\dot\gamma_k\cdot\lambda^1)(A)=\int_{\{\gamma\in A\}} \dot\gamma_k d\lambda^1.
\]
The mapping $\gamma\mapsto[\gamma]$ fails to be injective for two different reasons: On the one hand, the curve integral is invariant under
reparametrizations, i.e., $[\gamma]=[\gamma\circ h]$ for increasing differentiable bijections $h:[\alpha,\beta]\to[a,b]$. This non-uniqueness could be
remedied by considering suitable equivalence classes of curves. On the other hand, parts of the curve integrals could cancel if $\gamma$ 
\emph{goes back and forth} along some part of the curve, for example, the curve defined by $\gamma(t)=t$ for $t\in [0,1]$ and $\gamma(t)=2-t$ for
$t\in [1,2]$ yields the $1$-dimensional charge $[\gamma]=0$. This phenomenon makes it difficult to determine the variation measure
$\vnorm{[\gamma]}$ in general. For open sets $E$ we get from the description of $\vnorm\mu(E)$ as the supremum of $\<\mu,\varphi\>$
with $\varphi\in \DD(E)^n$ with $|\varphi|\le 1$ that
\[
\vnorm{[\gamma]}(E)\le \int_{\{\gamma\in E\}} |\dot\gamma| d\lambda^1
 \text{ and } \Var([\gamma])\le L(\gamma)
\]
where $L(\gamma)=\int_{[a,b]}|\dot\gamma|d\lambda^1$ is the length of $\gamma$. 
Outer regularity of finite Borel measures on $\R^n$ yields the same inequality for all Borel sets $E\subseteq\R^n$. Moreover,
if
$\Var([\gamma])=L(\gamma)$ we can pass to complements and we thus get
\begin{align}\label{eom}
 \Var([\gamma])= L(\gamma) \text{ if and only if } \vnorm{[\gamma]}(A)=\int_{\{\gamma\in A\}}|\dot\gamma| d\lambda^1
 \text{ for all $A\in\B^n$.}
\end{align}

One can easily calculate the divergence of $[\gamma]$:
Since Lipschitz curves are absolutely continuous, i.e., they satisfy the fundamental theorem of calculus, we have
\[
\Div([\gamma])(\psi)=-\int_{[a,b]}\< \nabla \psi\circ\gamma,\dot\gamma\>d\lambda^1 =-(\psi\circ\gamma)\big|_a^b
= (\delta_{\gamma(a)}-\delta_{\gamma(b)})(\psi).
\]

\medskip
For fixed $\ell>0$ (in most situations, the case $\ell=1$ is enough) we denote
\[
K_\ell  =\{\gamma:[0,\ell]\to\R^n: \gamma \text{ is Lipschitz with } |\dot\gamma|\le 1 \text{ $\lambda^1$-a.e.}\}.
\]
Every rectifiable curve of length $\ell$ can be reparametrized to become an element of $K_\ell$ even with $|\dot\gamma|=1$ $\lambda^1$-a.e., 
but we do not need this arclength parametrization.
To give a precise meaning to the Smirnov decomposition $\mu=\int_{K_\ell} [\gamma] d\nu(\gamma)$ we consider the uniform
norm $\|\gamma\|=\sup\{|\gamma(t)|:t\in [0,\ell]\}$ on $K_\ell\subseteq C([0,\ell],\R^n)$. The Arzel\`a-Ascoli theorem implies that
$\{\gamma\in K_\ell: |\gamma(0)|\le m \}$ are compact in $C([0,\ell],\R^n)$ so that $K_\ell$ is $\sigma$-compact. It is convenient to
compactify the set of curves: We consider a one-point compactification $i:\R^n\to\hat\R^n=\R^n\cup\{\infty\}$, e.g., the inverse stereographic
projection $i:\R^n\to S^n=\{x\in\R^{n+1}:|x|=1\}$ with the north pole $\infty$, and $I:C([0,\ell],\R^n)\to C([0,\ell],\hat\R^n)$, 
$\gamma\mapsto i\circ\gamma$ where the range has again the topology of uniform convergence. 
From $|\gamma(t)|\le |\gamma(0)|+\ell$ for $\gamma\in K_\ell$ it easily follows that every $\gamma\in \hat K_\ell=\overline{I(K_\ell)}$ is either constant
with value $\infty$ or belongs to $I(K_\ell)$. Again, the Arzel\`a-Ascoli theorem implies the compactness of
$\hat K_\ell$ which is convenient to get finite Borel measures $\nu$ on $K_\ell$ from Alaoglu's theorem combined with 
the Riesz-Kakutani representation theorem for the dual of $C(\hat K_\ell)$.

Integrals $\mu=\int_{K_\ell} [\gamma] d\nu(\gamma)$ are meant in a weak sense, i.e., for continuous fields $\varphi=(\varphi_1,\ldots,\varphi_n)$
with compact support we have
\[
\<\mu,\varphi\>=\int_{K_\ell} \< [\gamma],\varphi\> d\nu(\gamma).
\]
For this integral  to be well-defined we need the Borel measurability of $\gamma\mapsto \<[\gamma],\varphi\>$. The continuity of $\varphi$ yields
that the curve integrals defining  $\<[\gamma],\varphi\>$ are limits of Riemann-Stieltjes sums 
\[
\sum_{k=1}^{\ell m} \<\varphi(\gamma(k/m)),\gamma(k/m)-\gamma((k-1)/m)\>
\]
 for $m\to\infty$, and for fixed $m$, these sums depend
continuously on $\gamma\in K_\ell$ with respect to the uniform norm. Since indicator functions of open sets are increasing limits of continuous
functions a monotone class argument  yields the measurability of $\gamma\mapsto \<[\gamma],\varphi\>$ for 
all $\varphi$ with bounded measurable components and that the
displayed equality above extends to such vector fields.

We can now state the main theorems from \cite{Smi} (the first one together with proposition \ref{ThA2} below is Smirnov's Theorem A and the
second one is a version of his Theorem C):

\begin{theorem}\label{ThA}
 For every $\ell>0$ and every $n$-dimensional charge $\mu$ with $\Div(\mu)=0$ there is a finite positive measure $\nu$
on $K_\ell$ with
\[
\mu=\int_{K_\ell} [\gamma]d\nu(\gamma),\; \vnorm\mu=\int_{K_\ell} \vnorm{[\gamma]} d\nu(\gamma),
\]
and $\nu$-almost every $\gamma\in K_\ell$ has values in $\supp(\mu)$ and satisfies $\Var([\gamma])=L(\gamma)=\ell$.
\end{theorem}

\begin{theorem}\label{ThC}
 An $n$-dimensional charge $\mu$ has a decomposition $\mu=\int_{K_\ell} [\gamma]d\nu(\gamma)$ for some (or, equivalently, all) $\ell>0$
 and a finite positive measure $\nu$ on $K_\ell$ if and only if $\Div(\mu)$ is a real signed measure. Then there is a decomposition such that
 $\vnorm\mu=\int_{K_\ell} \vnorm{[\gamma]} d\nu(\gamma)$
and $\nu$-almost every $\gamma\in K_\ell$ has values in $\supp(\mu)$ and satisfies $\Var([\gamma])=L(\gamma)$.
\end{theorem}

The difference to the divergence free case is thus that the length of the representing curves is only bounded by $\ell$.
Before turning to the proofs of these theorems let us show the application mentioned in the introduction (where $\|\cdot\|_M$
denotes the uniform norm on $M$).

\begin{theorem}[Havin-Smirnov]\label{ThHS}
    Let $M\subseteq\R^n$ be a compact set such that all rectifiable curves in $M$ are constant. Then, for any $f,f_1,\ldots,f_n \in C(M)$ and any $\e>0$, 
    there is $\psi \in \DD(\R^n)$ such that
    \[
        \|f-\psi\|_M+\sum_{j=1} ^n \|f_j- \partial_j \psi\|_M<\e.
    \]
\end{theorem}    

The somewhat surprising aspect of this theorem about \emph{free approximation} is that although $M$ can be quite big, e.g., in terms of its Hausdorff dimension,
knowledge of $\partial_j\psi|_M$ does not yield any estimates for $\psi$ on $M$ which one would get in presence of rectifiable curves
in $M$: If $\gamma$ is a curve in $M$ joining $x,y$, the mean value inequality allows to estimate $|\psi(x)-\psi(y)|$ by
the norms  of $\partial_j\psi$ times the length of $\gamma$ (this argument shows the necessity of the absence of
rectifiable curves in $M$ for the free approximation). A similar application to the density of restrictions
$\{\psi|_M: \psi\in\DD(\R^n)\}$ in a suitably defined space $C^1(M)$ of continuously differentiable functions can be seen in
\cite{FLW} which was the main motivation for the present authors to get a better understanding of Smirnov's results.

\begin{proof}[Proof of Theorem \ref{ThHS}]
 We show with the aid of the Hahn-Banach theorem that $I:\mathcal{D}(\R^n)\longrightarrow C(M)^{n+1}$ defined by 
 $\psi\mapsto (\psi|_M,\partial_1\psi|_M\ldots\partial_n\psi|_M)$ has dense range. 
 For $\Phi$ in the dual of $C(M)^{n+1}$ with $\Phi\circ I= 0$ we thus have to show $\Phi=0$. 
 By the Riesz-Kakutani representation theorem there are (signed) Borel measures $\mu_0, \ldots,\mu_n$ on $M$ with
    \[
        \Phi (f_0,\ldots,f_n)= \sum_{j=0}^n \int_M f_j d\mu_j
    \]
    for all $f_0,\ldots,f_n\in C(M)$.
 For the charge $\mu=(\mu_1,\ldots,\mu_n)$ and $\psi\in \mathcal{D}(\R^n)$, we have
    \[
        0=\Phi (I(\psi))=\int_M \psi d\mu_0 + \sum_{j=1}^n\int_M \partial_j\psi d\mu_j
        = \<\mu_0 - \Div (\mu),\psi\>.
    \]
 Thus, $\Div(\mu)=\mu_0$ so that Theorem \ref{ThC} (for fixed $\ell>0$) yields a decomposition
    \[
        \mu = \int _{K_\ell} [\gamma] d\nu(\gamma)
    \]
where $\nu$-almost all curves have values in $M$.    Since by assumption, rectifiable curves in $M$ are constant
and hence satisfy $[\gamma]=0$,
we obtain
$\mu=0$, $\mu_0=\Div(\mu)=0$ and thus $\Phi=0$.
\end{proof}


\section{Proof for divergence free charges}

A main ingredient in this case is a version of Liouville's theorem that flows of divergence free smooth vector fields are volume preserving.
Let us recall that a linearly bounded vector field $\phi=(\phi_1,\ldots,\phi_n)\in C^1(\R^n)^n$, i.e., $|\phi(x)|\le c(1+|x|)$ for some constant
$c>0$ and all $x\in\R^n$, has a globally defined flow $u\in C^1(\R\times\R^n)$, i.e., $u(\cdot,x)$ is the unique solution of the initial
value problem $\partial_tu(t,x)=\phi(u(t,x))$ and $u(0,x)=x$. Depending on cultural context, this theorem is called after
some subset of $\{$Cauchy, Lipschitz, Picard, Lindel\"of$\}$. For a proof which deduces existence as well as the differentiability of $u$
from the inverse function theorem in Banach spaces, we refer to the book of Duistermaat and Kolk \cite[appendix B]{DuKo}.
The flow equation $u(t+s,x)=u(t,u(s,x))$ implies that all $u(t,\cdot):\R^n\to\R^n$ are diffeomorphisms with inverse $u(-t,\cdot)$.

\medskip
Every locally finite positive Borel measure $\varrho$ on $\R^n$ then yields a local charge $\phi\cdot \varrho$ which has a distributional divergence.
For $\varrho=\lambda^n$ we have $\Div(\phi\cdot\lambda^n)= \Div(\phi)\cdot\lambda^n$ where $\Div(\phi)=\sum_{k=1}^n\partial_k\phi_k$ is the classical
divergence of $\phi$.

\begin{proposition}[Liouville's theorem]
Let $\phi\in C^1(\R^n)^n$ be a linearly bounded vector field with corresponding flow $u$. Then every locally finite positive
Borel measure $\varrho$ with $\Div(\phi\cdot\varrho)=0$ is invariant under $u$, i.e., $\varrho_t=u(t,\cdot)_{\#}(\varrho)$ is independent of $t\in\R$.
\end{proposition}

\begin{proof}
    For $\varphi\in \mathcal{D}(\R^n)$, it is enough to show that
    \[
    \int\varphi d\varrho_t= \int\varphi(u(t,x))d\varrho(x)
    \]
is independent of $t$ which we will show by differentiation of this parameter integral for $\psi(t,x)=\varphi(u(t,x))$. 
The flow equation yields $\psi(t,u(-t,x))=\varphi(x)$, and hence
    \[
        0=\partial_t\left(\psi(t,u(-t,x))\right)= \partial_t \psi(t,u(-t,x)) -\< \nabla_x\psi(t,u(-t,x)), \phi (u(-t,x))\>.
    \]
    The surjectivity of $u(-t,\cdot)$ thus implies for all $(t,y)\in\R\times\R^n$
      \[
     \partial_t\psi(t,y)= \< \nabla_y\psi(t,y),\phi (y)\>.
     \]

    To justify the differentiability under the integral we fix a compact interval $J=[-m,m]$ and show that $\partial_t\psi(t,x)$ is dominated by some
    $\rho$-integrable function $g(x)$ which is independent of $t\in J$.   
 For $t\in J$, we have
    \[
       \supp (\psi(t,\cdot))\subseteq \{x\in\R^n:u(t,x)\in \supp(\varphi)\} =u(-t,\supp (\varphi)) \subseteq u(J\times \supp(\varphi)),
    \]
 where this last set $K$ is compact by the continuity of $u:\R\times\R^n\to\R^n$. The continuity of $\partial_t\psi$ thus implies that
 $|\partial\psi_t(t,\cdot)|$ is bounded by a constant (independent of $t$) times the indicator function of $K$ which is
 $\varrho$-integrable.
 
We can thus differentiate with respect to $t\in J$ under the integral and get
    \begin{align*}
        \partial_t\int \psi(t,x)d\varrho(x)  &= \int \partial_t\psi(t,x) d\varrho(x) 
         =\int \<\nabla_x\psi(t,x), \phi (x)\> d\varrho (x)
        \\ & =-\<\Div(\phi\cdot\varrho),\psi(t,\cdot)\>=0.\qedhere
    \end{align*}
    \end{proof}

\begin{proof}[Proof of theorem \ref{ThA}]  
    Let $\mu=(\mu_1,\ldots,\mu_n)\neq 0$ be a divergence free charge and $\ell>0$. We take a smooth and strictly positive function $k:\R^n\to\R$
    with integral $1$, e.g., a Gaussian kernel, and define the convolution $\tilde\mu=\mu * k=(\mu_1*k,\ldots,\mu_n*k)$ which is a smooth vector field
    with $|\mu*k|\le \vnorm\mu *k$ on $\R^n$. Since $\vnorm\mu$ is a positive measure and $k$ is strictly positive also $\vnorm\mu *k$ is strictly positive
    and  we can define $\phi=\mu*k/ (\vnorm{\mu}*k)$. This is a smooth vector field which is even bounded and hence linearly bounded
    so that the proposition applies to the flow $u$ of $\phi$ and the measure $\varrho=(\vnorm\mu*k)\cdot\lambda^n$ (which is finite as the convolution of
    $\vnorm\mu$ with the probability measure $k\cdot\lambda^n$ so that $\varrho(\R^n)=\Var(\mu)$) because
    \[
    \Div(\phi\cdot \varrho)= \Div((\mu*k)\cdot\lambda^n)=(\Div(\mu)*k)\cdot\lambda^n=0.
    \]


%


We will show now that $\tilde\mu=\mu*k$ decomposes into the curves $\gamma_x:[0,\ell]\to\R^n$, $t\mapsto u(t,x)$. 
Indeed, for $\varphi\in\mathcal{D}(\R^n)^n$, the invariance of $\varrho$ under $u(t,\cdot)$ yields
    \begin{align*}
        \int_{\R^n} [\gamma_x] (\varphi)d\varrho(x) &= \int_{\R^n} \int_0^\ell \<\varphi(u(t,x)),\partial_tu (t,x)\> dt d\varrho(x)
        \\
         &= \int _0 ^\ell \int_{\R^n} \< \varphi(u(t,x)),\phi(u(t,x))\> d \varrho(x) dt
         \\
        &= \int _0 ^\ell \int_{\R^n} \< \varphi,\phi\> d \varrho dt =\ell \int_{\R^n} \<\varphi, \phi\> (\vnorm{\mu}*k) d \lambda^n
        \\
        &= \ell \<\tilde\mu,\varphi\>.
    \end{align*}
  Form $|\phi|\le 1$ we get $\gamma_x\in K_\ell$ so that the image measure $\nu$ of $\varrho/\ell$ under the map 
  $\Gamma:\R^n\to K_\ell$, $x\mapsto \gamma_x$
  satisfies $\nu(K_\ell)=\Var(\mu)/\ell$ and 
  \[
  \tilde\mu=\int_{K_\ell}[\gamma] d\nu(\gamma).
  \]

 For $\e\in(0,1)$ we set $k_\e (x)=\frac1{\e^n} k(\frac{x}{\e})$ and obtain measures $\nu_\e$ on $K_\ell$ with 
 $\nu_\e(K_\ell)=\Var(\mu)/\ell$ and
    \[
        \mu*k_\e =\int _{K_\ell} [\gamma] d \nu_\e (\gamma). 
    \]
For $\eps\to 0$ we have $\mu*k_\e\to \mu$ in the sense of distributions, i.e., $\<\mu*k_\e,\varphi\>\to\<\mu,\varphi\>$ for all $\varphi\in\DD(\R^n)^n$.
On the other hand we consider $\nu_\e$ as measures on $\hat K_\ell$ with $\nu_\e(\{\infty\})=0$. Then Alaglou's (or, in this
context, Prokhorov's) theorem yields a sequence $\eps_j\to 0$ and a positive measure $\nu$ on $\hat K_\ell$ with $\nu_{\e_k}\to \nu$ in the
weak$^*$-topology, i.e., $\int fd\nu_{\e_k}\to\int f d\nu$ for all $f\in C(\hat K_\ell)$.

For each $\varphi\in \DD(\R^n)^n$, we will show 
\[
\<\mu,\varphi\> =\int_{K_\ell} [\gamma](\varphi) d\nu(\gamma).
\]
This follows from the weak$^*$-convergence if we prove the continuity of 
$\hat K_\ell\to\R$, $\gamma\mapsto [\gamma](\varphi)=\int_{[0,\ell]}\<\varphi\circ\gamma,\dot\gamma \> d\lambda^1$
(which is $0$ for the curve with constant value $\infty$ since $\varphi$ has compact support).

For a uniformly convergent sequence $\gamma_j\to\gamma$ it is enough to find a subsequence $j(i)$ with $[\gamma_{j(i)}](\varphi)\to[\gamma](\varphi)$.
The case of the curve with constant value $\infty$ is clear since then $\varphi\circ\gamma_j=0$ for $j$ big enough. Otherwise,
the sequence $\gamma_j$ is uniformly bounded. Since $|\dot\gamma_j|\le 1$ the sequence $\dot\gamma_j$ is bounded in the 
Hilbert space $L_2([0,\ell])^n$ and has thus a weakly convergent subsequence $\dot\gamma_{j(i)}\to g$ for some $g\in L_2([0,\ell])^n$.
Since $\gamma_j$ are absolutely continuous we obtain for $t\in [0,\ell]$
\[
\gamma(t)=\lim_{i\to\infty} \gamma_{j(i)}(0)+\int_{[0,t]} \dot\gamma_{j(i)} d\lambda^1 = \gamma(0)+\int_{[0,t]} g d\lambda^1,
\]
which implies $\dot\gamma=g$ $\lambda^1$-a.e. We therefore obtain
\[
[\gamma_{j(i)}](\varphi)=\int_{[0,\ell]} \<\varphi\circ\gamma_{j(i)}-\varphi\circ\gamma,\dot\gamma_{j(i)}\>d\lambda^1+\int_{[0,\ell]}
\< \varphi\circ\gamma,\dot\gamma_{j(i)}\> d\lambda^1 \to [\gamma](\varphi)
\]
because the first integrand converges uniformly to $0$ and the second integral converges because of the weak convergence in  $L_2([0,\ell])^n$ of
$\dot\gamma_{j(i)}\to\dot\gamma$.

\medskip
We will next show that $\vnorm\mu=\int\vnorm{[\gamma]}d\nu(\gamma)$ and that $\nu$-almost all curves have length $\ell$
(till know we only know that their length is $\le\ell$). For an open set $E\subseteq\R^n$ and $\varphi\in\DD(E)^n$ with $|\varphi|\le 1$ we have
\[
\<\mu,\varphi\>=\int_{K_\ell} [\gamma](\varphi)d\nu(\gamma)\le \int_{K_\ell} \vnorm{[\gamma]}(E) d\nu(\gamma)
\]
and hence $\vnorm\mu(E)\le \int_{K_\ell} \vnorm{[\gamma]}(E) d\nu(\gamma)$. For $E=\R^n$ this yields
\[
\Var(\mu) \le \int_{K_\ell} \Var{[\gamma]} d\nu(\gamma)\le \int_{K_\ell} L(\gamma)d\nu(\gamma)\le \ell\nu(K_\ell)\le\Var(\mu)
\]
so that none of these inequalities can be strict. This implies $\Var{[\gamma]}=L(\gamma)=\ell$ for $\nu$-almost all $\gamma\in K_\ell$, $\nu(K_\ell)=\Var(\mu)/\ell$,
and also that $\vnorm\mu=\int\vnorm{[\gamma]}d\nu(\gamma)$ by the same argument as in (\ref{eom}) (two positive
Borel measures $\rho$ and $\varrho$ with $\rho(E)\le\varrho(E)$ for all open sets and $\rho(\R^n)=\varrho(\R^n)<\infty$ are equal).

To show finally that $\nu$-almost every curve has values in the support $S$ of $\mu$ we first recall that
\[
\vnorm{[\gamma]}(E) \le \int_{\{\gamma\in E\}} |\dot\gamma|d\lambda^1.
\]
For $\gamma\in K_\ell$ with $\Var([\gamma])=L(\gamma)=\ell$ we thus get 
$|\dot\gamma|=1$ $\lambda^1$-a.e. and equality in the previous line. For $E=S^c$ we obtain
\[
0=\vnorm\mu (E) =\int_{K_\ell} \lambda^1(\{\gamma\in E\}) d\nu(\gamma)
\]
which implies $\lambda^1(\{\gamma\in E\})=0$ for $\nu$-almost all $\gamma$. The \emph{open} sets $\{\gamma\in E\}$ are thus empty
for $\nu$-almost all $\gamma$.
\end{proof}

The following additional property of the decomposition is not needed for the proof of theorem \ref{ThC} or applications in approximation theory.

\begin{proposition}
\label{ThA2}
    The decomposition in theorem \ref{ThA} also satisfies
    \[
        \vnorm{\mu}= \ell \int_{K_\ell} \delta_{\gamma(0)} d\nu(\gamma)= \ell \int_{K_\ell} \delta_{\gamma(\ell)} d\nu(\gamma).
    \]
\end{proposition}

\begin{proof}
We exploit the construction of $\nu$ as a weak limit of $\nu_\e$ defined as the push forward $\Gamma_\#(\varrho_\e/\ell)$ for
 $\varrho_\e=(\vnorm\mu*k_\e)\cdot\lambda^n$
under the map $\Gamma(x)=u(\cdot,x)$. For the evaluation $e_0:K_\ell\to\R^n$, $\gamma\mapsto\gamma(0)$ we thus have
$e_0\circ \Gamma(x)=u(0,x)=x$. 

For an open set $E$ and $\varphi\in \DD(E)^n$ with $|\varphi|\le 1$ we have 

    \begin{align*}
       \<\mu,\varphi\> &=  \lim _{k\to\infty} \<\mu_{\e_k},\varphi\> \le 
        \limsup_{k\to\infty} \int \pnorm{\varphi(x)} d\vnorm{\mu_{\e_k}}(x)   
         \le \limsup_{k\to\infty} \int \pnorm{\varphi(x)} d\varrho_{\e_k}(x)  \\ &
         = \limsup_{k\to\infty} \int \pnorm{\varphi\circ e_0\circ\Gamma} d\varrho_{\e_k}  
         =\ell \limsup_{k\to\infty} \int \pnorm{\varphi \circ e_0} d\nu_{\e_k}  
        \\ & =\ell \int |\varphi\circ e_0| d\nu 
         \le \ell\int \chi_E\circ e_0 d\nu =\ell \int \delta_{\gamma(0)}(E) d\nu(\gamma).
    \end{align*}
Taking the supremum we thus get $\vnorm\mu (E) \le \ell \int \delta_{\gamma(0)}(E)d\nu(\gamma)$. For $E=\R^n$ we have equality
because $\nu(K_\ell)=\Var(\mu)/\ell$ and by the outer regularity argument used already twice we conclude that we even have
equality for all Borel sets.

 The second formula in the proposition follows from
    \[
        0=\Div(\mu)= \int \Div ([\gamma]) d\nu (\gamma)=\int (\delta_{\gamma(0)}-\delta_{\gamma(\ell)} )d\nu (\gamma). \qedhere
    \]
\end{proof}


\section{Charges with the divergence being a measure}

In this section we prove theorem \ref{ThC} following Smirnov's arguments more closely than in the previous section. 
The necessity part follows from $\Div(\int_{K_\ell}[\gamma]d\nu(\gamma))=\int_{K_\ell} (\delta_{\gamma(0)}-\delta_{\gamma(\ell)})d\nu(\gamma)$.

Let now $\mu=(\mu_1,\ldots,\mu_n)$ be an 
$n$-dimensional charge whose divergence is a signed measure. We define an $(n+1)$-dimensional charge
\[
\mu^+=\big(\mu\otimes (\delta_0-\delta_\ell), -\Div(\mu)\otimes \lambda^1|_{[0,\ell]}\big).
\]
The product of a charge and a signed measure is meant component-wise and the product of two signed measures can be defined
using the polar decomposition: For $\varrho= r\cdot \vnorm\varrho$ and $\sigma=s\cdot\vnorm\sigma$ the product is
$\varrho \otimes \sigma=(r\otimes s)\cdot (\vnorm\varrho\otimes\vnorm\sigma)$ with $r\otimes s(x,t)=r(x)s(t)$.

The restriction of $\mu^+$ to the plane $\R^n\times\{0\}$ gives back $\mu$, i.e., for $A\in\B^n$ and the projection $\pi:\R^{n+1}\to\R^n$ onto the first $n$ components we have
\[
\pi\circ\mu^+(A\times\{0\})=\mu(A).
\]

Denoting the points of $\R^{n+1}$ as $(x,t)$ with $x\in\R^n$ and $t\in\R$ we write
$\nabla \psi=(\nabla_x \psi,\partial_t \psi)$. It is easy to see that $\mu^+$ is divergence free: For $\psi\in\DD(\R^{n+1})$ we have
    \begin{align*}
        -\<\Div (\mu^+),\psi\> &= \<\mu^+,\nabla \psi\> 
        = \<\mu\otimes (\d_0-\d_\ell),\nabla_x \psi\> - \<\Div(\mu) \otimes \lambda^1|_{[0,\ell]},\partial_t\psi\>
        \\
        &= \<\mu,\nabla_x \psi(\cdot,0)-\nabla_x \psi(\cdot,\ell)\> - \<\Div(\mu),  \int_{[0,\ell]} \partial_t\psi(\cdot,s)d\lambda_1(s)\>
        \\
        &= \<\Div(\mu),\psi (\cdot,\ell) -\psi(\cdot,0)\> -\<\Div(\mu),\psi (\cdot,\ell) -\psi(\cdot,0)\> =0.
    \end{align*}
 Theorem \ref{ThA} yields a finite positive measure $\nu^+$ on $K_\ell^+$ (the analogue of $K_\ell$ for curves in
 $\R^{n+1}$) with
    \[
        \mu^+=\int_{K_\ell^+} [\gamma] d\nu^+ (\gamma),\quad
        \vnorm{\mu^+}
        =\int_{K_\ell^+} \vnorm{\gamma} d\nu^+(\gamma),
    \]
and $\nu^+$-almost all $\gamma\in K_\ell^+$ satisfy $\Var([\gamma])=L(\gamma)=\ell$.

In order to get a representation of $\mu$ we will show for $\nu^+$-almost all curves $\gamma$ that $\{\gamma\in \R^n\times\{0\}\}$
is an interval $[\alpha,\beta]$ and that the curves $\tilde\gamma(t)=\gamma(\max\{\alpha,\min\{t,\beta\}\})$ (which rest at $\gamma(\alpha)$
for $t\in [0,\alpha]$, then follow $\gamma$ till they stay at $\gamma(\beta)$)
represent $\mu$.
We first analyse the behaviour of the curves in the open \emph{strip} $S=\R^n\times (0,\ell)$:
We have 
\[
\mu^+|_S=(0,-\Div(\mu)\otimes \lambda^1|_{(0,\ell)}) \text{ and } \vnorm{\mu^+|_S}=\vnorm{\Div(\mu)}\otimes\lambda^1|_{(0,\ell)}
\]
so that the polar decomposition is $\mu^+|_S= f\cdot \vnorm{\Div(\mu)}\otimes\lambda^1|_{(0,\ell)}$ where $f(x,t)=(0,g(x))$ with
$-\Div(\mu)=g\cdot \vnorm{\Div(\mu)}$. We set
\[
E_\pm=\{x\in\supp(\Div(\mu)): g(x)=\pm 1\}
\]
so that $\supp(\Div(\mu))=E_+\cup E_-$ because $|g(x)|=1$. 

We claim for $\nu^+$-almost all curves and $\lambda^1$-almost all $t\in [0,\ell]$
that $\gamma(t)\in E_\pm$ implies $\dot\gamma_{n+1}(t)=\pm 1$.

Let $\chi_+$ be the indicator function of the set $E_+ \times (0,\ell)$. Since the last component of $\dot\gamma$ satisfies 
$\dot\gamma_{n+1}\in [-1,1]$ for $\gamma\in K_\ell^+$ we have 
    \begin{align*}
        \nu^+ \otimes & \lambda^1|_{(0,\ell)} (\{  (\gamma,t):  \gamma(t)\in  E_+\times (0,\ell)\}) 
        = \int _{K_\ell^+\times (0,\ell)} \chi_+ (\gamma (t)) d \nu^+ \otimes \lambda^1|_{(0,\ell)}(\gamma,t)   \\
        &\geq \int _{K_\ell^+} \int _{[0,\ell]} (\chi_+ \circ \gamma) (t) \dot\gamma_{n+1} (t) d\lambda^1(t) d\nu^+ (\gamma) 
       =\int _{K_\ell^+}\< [\gamma], (0,\ldots,0,\chi_+)\> d \nu^+ (\gamma)        \\
        &= \<\mu^+,(0,\ldots,0,\chi_+)\>
        = \int_{\R^{n+1}} \chi_+\ d (-\Div(\mu) \otimes \lambda^1|_{(0,\ell)}) \\
        &=  \int_{\R^{n+1}} \chi_+  g d(\vnorm{\Div(\mu)} \otimes  \lambda^1|_{(0,\ell)})
        = \vnorm{\mu^+} (E_+\times (0,\ell)) \\
        &= \int _{K_\ell^+} \vnorm{[\gamma]} (E_+\times (0,\ell)) d \nu^+ (\gamma)
        = \int _{K_\ell^+}  \lambda^1(\gamma ^{-1}  (E_+\times (0,\ell)))   d \nu^+ (\gamma)\\
        &=\nu^+ \otimes  \lambda^1|_{(0,\ell)} (\{  (\gamma,t) :  \gamma(t)\in  E_+\times (0,\ell)\}),
    \end{align*}
    where the next-to-last equality comes from (\ref{eom}). 
    It follows that the only inequality above is also an equality and this implies that, for $\nu^+$-almost all $\gamma$, we have
    $\chi_+(\gamma(t))\dot \gamma_{n+1} (t) =\chi_+(\gamma(t))$ for $\lambda^1$-almost all $t\in (0,\ell)$.
    The same arguments show the analogous statement for $E_-$ so that we have proved the claim.
    

   

    Since curves in $K_\ell^+$ are absolutely continuous with $|\dot\gamma|=1$ $\lambda^1$-a.e., we obtain that $\nu^+$-almost all curves are 
    \emph{vertical} when they pass through $S$, more precisely, $\gamma(t)\in E_\pm\times (0,\ell)$ implies
    $\gamma(s)=\gamma(t)\pm (s-t)e_{n+1}$ for the $(n+1)$th unit vector $e_{n+1}$ in $\R^{n+1}$
    and all $s\in J$ where $J\subseteq[t,\ell]$ is the maximal interval containing $t$ with $\gamma(J)\subseteq S$.

    We next define
    \[
   D=\{\gamma\in K_\ell^+: \gamma(0)\in (\R^n\times\{0\})\cup(E_-\times (0,\ell)\},
    \]
  $\alpha(\gamma)=\inf\{t\in [0,\ell]: \gamma(t)\in \R^n\times\{0\}\}$, and $\beta(\gamma)=\sup\{t\in [0,\ell]: \gamma(t)\in \R^n\times\{0\}\}$
  (with $\sup\emptyset=0$ and $\inf\emptyset=\ell$).
    For $\nu^+$ almost every curve $\gamma\in D$ we then have $\{\gamma\in\R^n\times\{0\}\}=[\alpha(\gamma),\beta(\gamma)]$ and
    $\dot\gamma(t)=(0,\ldots,0,\pm1)$ for $t\in [0,\ell]\setminus [\alpha(\gamma),\beta(\gamma)]$.
    
    For $A\in \B^n$, we thus get $[\gamma](A\times \{0\}) = ([\pi\circ\gamma](A),0)$, and since $\nu^+$-almost all curves in  $K_\ell^+\setminus D$
    do not meet the plane $\R^n\times\{0\}$ we obtain
    \begin{align*}
    (\mu,0)(A\times \{0\}) &=\mu^+|_{\R^n\times \{0\}}(A\times\{0\}) = \int_{D}  \<[\gamma],\chi_{A\times\{0\}}\> d\nu^+(\gamma)\\
    &=
    \int_{D}\<([\pi_H\circ\gamma],\chi_A\>,0)d\nu^+(\gamma).
    \end{align*}
    Taking the first $n$ coordinates we thus have
    \[
    \mu=\int_{D}[\pi_H\circ\gamma]d\nu^+(\gamma) =\int_{K_\ell}[\gamma]d\nu(\gamma)
    \]
where $\nu$ is the push forward of $\nu^+|_D$ under the mapping $\gamma\mapsto\tilde\gamma$ defined by
$\tilde\gamma(t)=\pi(\gamma(\max\{\alpha(\gamma),\min\{t,\beta(\gamma)\}\})$. Then $\nu$-almost all curves in $K_\ell$ have values in $\supp(\mu)$ and
satisfy $\Var([\gamma])=L(\gamma)$ because for $\nu^+$-almost all curves $\gamma$ the projections
$\pi\circ\gamma$ have this property.  Finally, we have $\vnorm\mu=\int_{K_\ell}\vnorm{[\gamma]}d\mu(\gamma)$
because of the corresponding property for $\mu^+$ and $\vnorm{\mu}(A)=\vnorm{\mu^+}(A\times \{0\})$.

    \section{The role of the length of the curves}
    
    In applications like theorem \ref{ThHS} or in \cite{FLW} it is enough to have theorem \ref{ThC} for a fixed $\ell$.
    In \cite{Smi}, curves of length $\ell\to\infty$ are used in two situations. In Theorem B of \cite{Smi} Smirnov shows
    a decomposition of divergence free charges into \emph{elementary solenoids} which are Lipschitz curves $\gamma:\R\to\R^n$
    such that the limits 
    \[
      \lim_{s\to\infty}\frac 1{2s} \int_{-s}^s \<\varphi(\gamma(s)),\dot\gamma(s)\>ds
    \]
    exist for all $\varphi\in \DD(\R^n)^n$. In contrast to theorem \ref{ThA} (where the curves of fixed length $\ell$ are hardly ever
    closed and thus have a non-vanishing divergence $\delta_{\gamma(0)}-\delta_{\gamma(\ell)}$) these elementary solenoids are divergence free. 
    Smirnov gives examples of even smooth three-dimensional charges which cannot be decomposed into closed curves.

    The other situation where curves of arbitrary length are used is Smirnov's Theorem C which is a slightly different version of
    theorem \ref{ThC}. Using the methods of the previous section he proves that every charge $\mu$ can be decomposed $\mu=\sigma+\tau$
   with a divergence free part $\sigma$ to which theorem \ref{ThA} applies and a charge $\tau$ which decomposes as
   $\tau=\int_K[\gamma] d\nu(\gamma)$ with $K=\bigcup_{\ell>0} K_\ell$ such that
   \[
   \vnorm\mu=\vnorm\sigma +\vnorm\tau,\, \vnorm\tau=\int_K\vnorm{[\gamma]}d\nu(\gamma), \text{ and }
   \vnorm{\Div(\mu)}=\int_K \vnorm{\Div([\gamma])}d\nu.
   \]
   This decomposition is achieved by a recursive approximation process which starts as $\mu=\mu_1+ \tau_1$ where -- 
    in the situation of the proof of theorem \ref{ThC} -- one can take
   \[
\tau_1= \int_{\tilde D} [\pi\circ\gamma]d\nu^+(\gamma)
\]
with the set $\tilde D=\{\gamma\in D: \gamma(\ell)\in E_+\times (0,\ell)\}$ and $\mu_1=\mu-\tau_1$. In the next step, this
procedure is applied to get $\mu_1=\mu_2+\tau_2$. Smirnov shows that one can choose curve lengths $\ell_k$
in such a way that the charges $\mu_k$ converge (in the total variation norm) to a divergence free charge $\sigma$
and that also the series $\tau=\sum_{k=1}^\infty \tau_k$ converges.



\bibliographystyle{amsalpha}
\bibliography{SmirnovDecompositions}

\end{document}